\documentclass[12pt]{amsart}

\usepackage{amsmath,amsthm,amsfonts,amssymb,mathrsfs}

\newtheorem{theorem_etc}{theorem_etc}[section]
\newtheorem{theorem}[theorem_etc]{Theorem}
\newtheorem{proposition}[theorem_etc]{Proposition}

\newtheorem{definition}[theorem_etc]{Definition}
\newtheorem{remark}[theorem_etc]{Remark}

\newcommand{\N}{\mathbb{N}}

\newcommand{\Z}{\mathbb{Z}}
\newcommand{\FP}{\operatorname{\mathit{FP}}}
\newcommand{\AP}{\operatorname{\mathit{AP}}}

\allowdisplaybreaks

\begin{document}
\title{Free paratopological groups}

\author{Ali Sayed Elfard}
\address{School of Mathematics and Applied Statistics,
University of Wollongong,
Wollongong, Australia}
\email{ae351@uow.edu.au\textrm{,} a.elfard@yahoo.com}



\begin{abstract}
Let $\FP(X)$ be the free paratopological group on a~topological space $X$ in the sense of Markov. In this paper, we study the group $\FP(X)$ on a~$P_\alpha$-space $X$ where $\alpha$ is an~infinite cardinal and then we prove that the group $\FP(X)$ is an~Alexandroff space if $X$ is an~Alexandroff space. Moreover, we introduce a~neighborhood base at the identity of the group $\FP(X)$ when the space $X$ is Alexandroff and then we give some properties of this neighborhood base. As applications of these, we prove that the group $\FP(X)$ is $T_0$ if $X$ is $T_0$, we characterize the spaces $X$ for which the group $\FP(X)$ is a~topological group and then we give a~class of spaces $X$ for which the group $\FP(X)$ has the inductive limit property.
\end{abstract}
\maketitle

\section{Introduction}
Let $\FP(X)$ and $\AP(X)$ be the free paratopological group and the free abelian paratopological group, respectively, on a~topological space~$X$ in the sense of Markov. The group $\FP(X)$ is the abstract free group $F_a(X)$ on~$X$ with the strongest paratopological group topology on~$F_a(X)$ that induces the original topology on~$X$ and the abelian group~$\AP(X)$ is the abstract free abelian group~$A_a(X)$ on~$X$ with the strongest paratopological group topology on~$A_a(X)$ that induces the original topology on~$X$. For more information about free paratopological groups, see~(\cite{Romaguera},~\cite{Pyrch1},~\cite{Elfard1},~\cite{Elfard2},~\cite{Elfard3}).

In 1937, P. Alexandroff~\cite{Alexandroff} introduced a class of topological spaces under the name of \textit{Diskrete R\"{a}ume (Discrete space)} which is a space in which an arbitrary intersections of open sets is open. Now the name has been changed to \textit{Alexandroff space} since the discrete space is a space in which every singleton set is open. Recently, researchers have shown an increased interest in studying Alexandroff spaces. This may be due to the important applications of Alexandroff spaces in some areas of mathematical sciences such as the field of computer science.

In this paper, we study the groups $\FP(X)$ and $\AP(X)$ on a $P_\alpha$-space $X$, where $\alpha$ is an infinite cardinal (a topological space $X$ is a \textit{$P_\alpha$-space}, where $\alpha$ is an infinite cardinal if the set $\bigcap \mathscr{C}$ is open in $X$ for each family $\mathscr{C}$ of open subsets of $X$ with $|\mathscr{C}|<\alpha$). Then in Theorem~\ref{Alexandroff4}, we prove that the groups $\FP(X)$ and $\AP(X)$ are Alexandroff spaces if the space $X$ is Alexandroff and in Theorem~\ref{N}, we introduce simple neighborhood bases at the identities of the groups $\FP(X)$ and $\AP(X)$ for their topologies. Moreover, we study the groups $\FP(X)$ and $\AP(X)$ in the case where $X$ is a partition space and in another case where $X$ is a  $T_0$ Alexandroff space.

As applications of these results, in~Theorem~\ref{topgp}, we characterize the spaces $X$ for which the paratopological groups $\FP(X)$ and $\AP(X)$ are topological groups and in~Theorem~\ref{T0}, we prove that the group $\FP(X)$ is $T_0$ if the space $X$ is $T_0$. Finally, in Theorem~\ref{induc}, we give a class of spaces $X$ for which the groups $\FP(X)$ and $\AP(X)$ have the inductive limit property.

The content of this paper is adapted from the author's thesis~\cite{Elfard3}, chapter 3. We remark that the results in~Theorem~\ref{topgp} and~Theorem~\ref{T0} were found independently by the author in his thesis~\cite{Elfard3}. However, similar to these results were found by Pyrch~(\cite{Pyrch2},~\cite{Pyrch3}).

 \section{Definitions and Preliminaries}
 \label{DePr}
A \textit{paratopological group} is a pair $(G, \mathcal{T})$, where~$G$ is a group and $\mathcal{T}$ is a topology on~$G$ such that the mapping $(x, y)\mapsto xy$ of $G\times G$ into~$G$ is continuous. If in addition, the mapping $x\mapsto x^{-1}$ of~$G$ into $G$ is continuous, then $(G, \mathcal{T})$ is a \textit{topological group}.

If $(G, \mathcal{T})$ is a paratopological group, then simply we denote it by~$G$.\\

Marin and Romaguera \cite{MaRo} described a complete neighborhood base at the identity of any paratopological group as follows:

\begin{proposition}
\label{basep}
Let $G$ be a group and let $\mathscr{N}$ be a collection of subsets of~$G$, where each member of~$\mathscr{N}$ contains the identity element $e$ of~$G$. Then the collection~$\mathscr{N}$ is a base at~$e$ for a paratopological group topology  on~$G$ if and only if the following conditions are satisfied:
\begin{enumerate}
\item[(1)] for all $U, V\in \mathscr{N}$, there exists $W\in \mathscr{N}$ such that $W\subseteq U\cap V$;
\item[(2)] for each $U\in \mathscr{N}$, there exists $V\in \mathscr{N}$ such that $V^{2}\subseteq U$;
\item[(3)] for each $U\in \mathscr{N}$ and for each $x\in U$, there exists $V\in \mathscr{N}$ such that $xV\subseteq U$ and $Vx\subseteq U$; and
\item[(4)] for each $U\in \mathscr{N}$ and each $x\in G$, there exists $V\in \mathscr{N}$ such that $xVx^{-1}\subseteq U$.\\
\end{enumerate}
\end{proposition}

\begin{definition}~\cite{Elfard1}
\label{freeM1}
\begin{rm}
Let $X$ be a subspace of a paratopological group~$G$. Suppose that
\begin{enumerate}
\item[(1)] the set $X$ generates $G$ algebraically, that is, $\langle X\rangle=G$ and
\item[(2)] every continuous mapping $f:X\to H$ of~$X$ to an arbitrary
paratopological group $H$ extends to a continuous homomorphism $\hat{f}:G\to H$.
\end{enumerate}
Then $G$ is called the \textit{Markov free paratopological group on~$X$}, and is denoted by~$\FP(X)$.\\
By substituting ``abelian paratopological group'' for each occurrence of ``paratopological group''  above we obtain the definition of the \textit{Markov free abelian paratopological group on~$X$} and we denote it by $\AP(X)$.
\end{rm}
\end{definition}
 \begin{remark}
\label{strongestpt}
\begin{rm}
We denote the free topology of $\FP(X)$ by $\mathcal{T}_{FP}$ and the free topology of $\AP(X)$ by $\mathcal{T}_{AP}$ and we note that the topologies $\mathcal{T}_{FP}$ and $\mathcal{T}_{AP}$ are the strongest paratopological group topologies on the underlying sets of $\FP(X)$ and $\AP(X)$, respectively, that induce the original topology on $X$.
\end{rm}
\end{remark}

\section{$P_\alpha$-spaces}
Let $X$ be a topological space and $\alpha$ be an infinite cardinal. We say that $X$ is a \textit{$P_\alpha$-space} if the set $\bigcap \mathscr{C}$ is open in $X$ for each family $\mathscr{C}$ of open subsets of $X$ with $|\mathscr{C}|<\alpha$. Let $\tau$ be the topology of $X$. Then we define the topology~$\tau_\alpha$ to be the intersection of all topologies~$\mathcal{O}$ on~$X$ where $\tau\subseteq\mathcal{O}$ and $(X, \mathcal{O})$ is a~$P_\alpha$-space. Since the discrete topology on~$X$ contains~$\tau$ and is a $P_\alpha$-space, $\tau_\alpha$ exists and $(X, \tau_\alpha)$ is a~$P_\alpha$-space. We call the topology~$\tau_\alpha$ the \textit{$P_\alpha$-modification of~$\tau$}.

We note that if $X$ is a $P_{\alpha}$-space, then $X$ is a $P_{\alpha^+}$-space, where $\alpha^+$ is the successor cardinal of $\alpha$.\\

For the remain of this section we assume that $\alpha$ is a fixed infinite cardinal unless we say otherwise.

\begin{theorem}
\label{base1}
Let $(X, \tau)$ be a topological space and let~$\alpha^+$ be the infinite successor cardinal of~$\alpha$.  Then the collection of all sets which are the intersection of fewer than~$\beta$ open subsets of~$X$ is a base for the topology $\tau_\alpha$ on~$X$, where $\beta=\alpha$ if $\alpha$ is regular and $\beta=\alpha^+$ if $\alpha$ is singular.
\end{theorem}

\begin{proof}
Let $\tau=\{U_i\}_{i\in I}$. We show that the collection $\mathscr{B}=\{\bigcap_{d\in D}U_d: D\subseteq I$ and $|D|< \beta\}$ of subsets of $X$ is a base for the topology $\tau_\alpha$ on $X$, where $\beta$ as defined in the statement of the theorem. It is well known that every infinite successor cardinal is regular, so in both cases, $\beta$ is regular.

We show first that $\mathscr{B}$ is a base for some topology $\tau^*$ on~$X$. If $x\in X$, then there exists $i_0\in I$ where~$x\in U_{i_0}$ and such that $U_{i_0}\in \mathscr{B}$. Let $B_1, B_2 \in \mathscr{B}$ and let $x\in B_1\cap B_2$. Assume that  $B_1=\bigcap_{d\in D}U_d$ and $B_2=\bigcap_{t\in T}U_t$, where $D, T\subseteq I$ and $|D|, |T|< \beta$. Let $R=D\cup T$. So $|R|< \beta$. Hence $B_3=\bigcap_{r\in R}U_r\in \mathscr{B}$ and $x\in B_3\subseteq B_1\cap B_2$. Therefore, $\mathscr{B}$ is a base for some topology $\tau^*$ on~$X$.

We show second that $(X, \tau^*)$ is a~$P_\alpha$-space. Let $\tau^*=\{V_j\}_{j\in J}$ and let $M\subseteq J$ where $|M|< \beta$. Then we have
\begin{align*}
\bigcap_{m\in M}V_m&=\bigcap_{m\in M}\bigcup_{i\in I_m}B_{m, i}\\
&=\bigcup_{f\colon M\to \bigcup_{m\in M}I_m, f(m)\in I_m \forall m\in M}(\bigcap_{m\in M}B_{m, f(m)})\in \tau^*,
\end{align*}
where $I_m$ is an index set and $B_{m, i}\in \mathscr{B}$ for all $m\in M$ and $i\in I_m$. Thus $\tau^*$ contains $\tau$ and $(X, \tau^*)$ is a $P_\beta$-space, which implies that in both cases of $\beta$, $(X, \tau^*)$ is a~$P_\alpha$-space.

Now let $\hat{\tau}$ be a topology on $X$ containing $\tau$ such that $(X, \hat{\tau})$ is a $P_\alpha$-space. Then in the case where $\alpha$ is regular, we have $\mathscr{B}\subseteq \hat{\tau}$ and in the case where $\alpha$ is singular, by the argument above, we have $(X, \hat{\tau})$ is a~$P_{\alpha^+}$-space, which implies that $\mathscr{B}\subseteq \hat{\tau}$. Thus~$\tau^*\subseteq \hat{\tau}$ and hence $\tau^*$ is the smallest topology on~$X$ containing $\tau$ such that~$(X, \tau^*)$ is a~$P_\alpha$-space. Therefore, $\tau^*=\tau_\alpha$.
\end{proof}

\begin{proposition}
\label{para}
 Let $(G, \tau)$ be a paratopological group. Then $(G, \tau_\alpha)$ is a paratopological group.
 \end{proposition}

\begin{proof}
Let $g_1, g_2\in G$ and let $U\in \tau_\alpha$ contains $g_1g_2$. By Theorem~\ref{base1}, there is a set $\Lambda$, where $|\Lambda|< \beta$ and $\beta$ is as in the theorem such that $g_1g_2\in \bigcap_{\lambda\in \Lambda}U_\lambda\subseteq U$  where $U_\lambda\in \tau$ for all $\lambda\in \Lambda$. Thus $g_1g_2\in U_\lambda$ for all $\lambda\in \Lambda$. Since $\tau$ is a paratopological group topology on~$G$, for each $\lambda\in \Lambda$, there are~$V(\lambda), W(\lambda)\in \tau$ containing~$g_1, g_2$, respectively, such
 that $V(\lambda)W(\lambda)\subseteq U_\lambda$. Let~$U_1=\bigcap_{\lambda\in \Lambda}V(\lambda)$ and~$U_2=\bigcap_{\lambda\in \Lambda}W(\lambda)$. Then $U_1U_2\subseteq U_\lambda$ for all~$\lambda\in \Lambda$. Hence, $U_1, U_2\in \tau_\alpha$ and $U_1U_2\subseteq \bigcap_{\lambda\in \Lambda}U_\lambda\subseteq U$.
  Therefore, $\tau_\alpha$ is a paratopological group topology on $G$.
\end{proof}

\begin{proposition}
\label{pspace}
 Let $X$ be a topological space. Then the group~$\FP(X)$ is a $P_\alpha$-space if and only if the space $X$ is a $P_\alpha$-space.
\end{proposition}
\begin{proof}

$\Longrightarrow$: It is easy to prove that $X$ is a $P_\alpha$-space.

$\Longleftarrow$: Let $\tau$ be the topology of $X$ and let $\mathcal{T}_{FP}$ be the free topology of $\FP(X)$. We show that
$(\mathcal{T}_{FP})_\alpha=\mathcal{T}_{FP}$. By Proposition~\ref{para}, $(\mathcal{T}_{FP})_\alpha$ is a paratopological group topology on $F_a(X)$ and it is stronger than $\mathcal{T}_{FP}$. However, $\mathcal{T}_{FP}$ is the free paratopological group topology on $F_a(X)$, which is the strongest paratopological group topology on $F_a(X)$ inducing the original topology $\tau$ on $X$. Since $(\mathcal{T}_{FP})_\alpha|_ X=(\mathcal{T}_{FP}|_ X)_\alpha$ and $(\mathcal{T}_{FP}|_ X)_\alpha=(\tau)_\alpha=\tau$,  $(\mathcal{T}_{FP})_\alpha$ induces the topology $\tau$ of $X$. Thus we have $(\mathcal{T}_{FP})_\alpha=\mathcal{T}_{FP}$ and therefore, $\FP(X)$ is a $P_\alpha$-space.
\end{proof}
The same result of Proposition~\ref{pspace} is true for $\AP(X)$.

\section{Free paratopological groups on Alexandroff spaces}

A topological space $X$ is said to be \textit{Alexandroff} \cite{Arenas} if the intersection of every family of open subsets of~$X$ is open in $X$.

We note that a topological space $X$ is Alexandroff if and only if $X$ is a~$P_\alpha$-space for every infinite cardinal~$\alpha$. Thus by using Proposition~\ref{pspace}, we get the next result.
\begin{theorem}
\label{Alexandroff4}
The group $\FP(X)$ ($\AP(X)$) on a~space $X$ is an Alexandroff space if and only if $X$ is an Alexandroff space.
\end{theorem}

Let $G$ be a group and let $H$ be a submonoid of $G$. Then we say that $H$ is a \textit{normal} submonoid of $G$ if $ghg^{-1}\in H$ for all $h\in H$ and $g\in G$.

\begin{proposition}
\label{submonoid}
 If $H$ is a normal submonoid of a group $G$, then $\{H\}$ is a neighborhood base at the identity of $G$ for a paratopological group topology on $G$.
\end{proposition}
\begin{proof}
Let $H$ be a normal submonoid of $G$. Then $\{H\}$ satisfies the conditions of Proposition~\ref{basep}, and therefore, $\{H\}$ is a neighborhood base at the identity of $G$ for a paratopological group topology on~$G$.
\end{proof}
Let $X$ be a topological space. For each $x\in X$, let $U(x)=\bigcap\{U\colon x\in U$ and $U$ is open$\}$. Then it is easy to see that the space $X$ is Alexandroff if and only if  the set $U(x)$ is open in $X$ for each $x\in X$.
%\begin{proposition}
%\label{single}
% Let $X$ be an Alexandroff space. Then the neighborhood base at the identity $e$ $(0_A)$ of $\FP(X)$ ($\AP(X)$) is a single normal submonoid.
%\end{proposition}
%\begin{proof}
%By Theorem~\ref{Alexandroff4}, the group $\FP(X)$ is an Alexandroff space. So $\{U(e)\}$ is a local base at $e$ for the free topology of $\FP(X)$. Now there exists a neighborhood $V$ of $e$ in $\FP(X)$ such that $V^2\subseteq U(e)$. Since $U(e)\subseteq V$, $U(e)^2\subseteq V^2\subseteq U(e)$. Therefore,  $U(e)$ is a submonoid.
%
%If $g\in \FP(X)$, then by condition (4) of Proposition~\ref{basep}, there exists a neighborhood $W$ of $e$  such that $gWg^{-1}\subseteq U(e)\subseteq W$, which implies that $gU(e)g^{-1}\subseteq gWg^{-1}\subseteq U(e)$. Therefore, $U(e)$ is a normal submonoid of $\FP(X)$. Similarly, we can prove the statement of the theorem for $\AP(X)$.
%\end{proof}

Let $X$ be an Alexandroff space and let $\FP(X)$ and $\AP(X)$ be the free paratopological group and the free abelian paratopological group, respectively, on $X$. Let $U_A=\bigcup_{x\in X}(U(x)-x)\subseteq\AP(X)$. Then we define $N_A$ to be the smallest submonoid of $\AP(X)$ containing the set $U_A$. So $N_A$ is of the form
\begin{center}$N_A=\{y_1-x_1+y_2-x_2+\cdots+y_n-x_n\colon x_i\in X, y_i\in U(x_i)$ for all $i=1, 2, \ldots, n, n\in \N\}$.\end{center}
Or simply, we write $N_A=\langle U_A\rangle$. Since every submonoid of an abelian group is normal, $N_A$ is normal. However, in this case, we will omit the word normal and say submonoid.

Let $\mathcal{N}_A=\{N_A\}$. Since $N_A$ is a submonoid of $\AP(X)$, by Proposition~\ref{submonoid}, $\mathcal{N}_A$ is a neighborhood base at the identity $0_A$ of $\AP(X)$  for a paratopological group topology $\mathcal{O}_A$ on $A_a(X)$.

Now for the group $\FP(X)$,  let $U_F=\bigcup_{x\in X}x^{-1}U(x)\subseteq \FP(X)$ and then we define $N_F$ to be the smallest normal submonoid of $\FP(X)$ containing the set $U_F$. The normal submonoid $N_F$ consists exactly of the set of all elements of the form,
$$w=g_1x_1^{-1}y_1g_1^{-1}\cdot g_2x_2^{-1}y_2g_2^{-1}\cdots g_nx_n^{-1}y_ng_n^{-1}$$
where $n\in \Bbb N$, $g_1, g_2, \ldots, g_n$ is an arbitrary finite system of elements of $F_a(X)$ and $x_1^{-1}y_1, x_2^{-1}y_2, \ldots, x_n^{-1}y_n$ is an arbitrary finite system of elements of $U_F$. Define $\mathcal{N}_F=\{N_F\}$. By Proposition~\ref{submonoid}, $\mathcal{N}_F$ is a neighborhood base at the identity $e$ of $\FP(X)$ for a paratopological group topology $\mathcal{O}_F$ on the free group $F_a(X)$.

\begin{proposition}
\label{finer1}
The topologies $\mathcal{O}_F$ and $\mathcal{O}_A$  induce topologies coarser than the original topology on $X$.
\end{proposition}

\begin{theorem}
\label{N}
The collection $\mathcal{N}_F$ ($\mathcal{N}_A$) is a neighborhood base at $e$ ($0_A$) for the free topology of $\FP(X)$ ($\AP(X)$).
\end{theorem}
\begin{proof}
We prove the theorem for $\mathcal{N}_F$, since the proof for $\mathcal{N}_A$ is the same. We show first that the topology $\mathcal{O}_F$ is finer than the free topology $\mathcal{T}_{FP}$ of $\FP(X)$. Let $\xi\colon X\to G$ be a continuous mapping of the space $X$ into an arbitrary paratopological group $G$. Then $\xi$ extends to a homomorphism $\hat{\xi}\colon F_a(X)\to G$. We show that $\hat{\xi}$ is continuous with respect to the topology $\mathcal{O}_F$. Let $V$ be a neighborhood of $\hat{\xi}(e)=e_G$ in $G$. Fix $x\in X$. Then $\xi(x)V$ is a neighborhood of $\xi(x)$ in $G$. Since $\xi$ is continuous at $x$, $\xi(U(x))\subseteq \xi(x)V$ and Since $\hat{\xi}|_X=\xi$,  $\hat{\xi}(U(x))\subseteq \hat{\xi}(x)V$. Because $\hat{\xi}$ is a homomorphism, $\hat{\xi}\big(x^{-1}U(x)\big)\subseteq V$. Since  $x$ is any point in $X$, we have
\begin{equation}
\label{eq4}
\hat{\xi}\big(\bigcup_{x\in X}x^{-1}U(x)\big)\subseteq V.
\end{equation}
Fix $n\in \N$. Then there exists a neighborhood $U$ of $e_G$ in $G$ such that $U^n\subseteq V$ and also, for all $g\in F_a(X)$, there exists a neighborhood $W$ of $e_G$ in $G$ such that $\hat{\xi}(g)W\big(\hat{\xi}(g)\big)^{-1}\subseteq U$. Since $V$ is any neighborhood of $e_G$ in $G$, from~(\ref{eq4}), we have $\hat{\xi} \big(\bigcup_{x\in X}x^{-1}U(x)\big) \subseteq W$. Fix $g\in F_a(X)$. So we have $$\hat{\xi}(g)\hat{\xi} \big(\bigcup_{x\in X}x^{-1}U(x)\big)\big(\hat{\xi}(g)\big)^{-1} \subseteq \hat{\xi}(g)W\hat{\xi}(g)^{-1}.$$

Since $\hat{\xi}$ is a homomorphism,
\begin{equation}
\label{eq5}
\hat{\xi} \big(\bigcup_{x\in X}gx^{-1}U(x)g^{-1}\big) \subseteq \hat{\xi}(g)W\hat{\xi}(g)^{-1}\subseteq U.
\end{equation}
  Since~(\ref{eq5}) holds for every $g\in F_a(X)$, we have
$$\hat{\xi} \big(\bigcup_{g\in F_a(X)}\bigcup_{x\in X}gx^{-1}U(x)g^{-1}\big) \subseteq U.$$
Thus we have $$\hat{\xi} \Big(\big(\bigcup_{g\in F_a(X)}\bigcup_{x\in X}gx^{-1}U(x)g^{-1}\big)^n\Big) \subseteq U^n\subseteq V.$$
Since $n$ is any element of $\N$,
       $$\hat{\xi} \Big(\bigcup_{n\in \N}\big(\bigcup_{g\in F_a(X)}\bigcup_{x\in X}gx^{-1}U(x)g^{-1}\big)^n\Big) \subseteq V.$$
     Since $N_F=\bigcup_{n\in \N}\big(\bigcup_{g\in F_a(X)}\bigcup_{x\in X}gx^{-1}U(x)g^{-1}\big)^n$, we have  $\hat{\xi}(N_F)\subseteq V$. Thus $\hat{\xi}$ is continuous with respect to the topology $\mathcal{O}_F$ and therefore, $\mathcal{O}_F$ is finer than $\mathcal{T}_{FP}$.
By Proposition~\ref{finer1}, $\mathcal{O}_F|_X$ is coarser than the original topology on $X$. Since $\mathcal{O}_F$ is finer that $\mathcal{T}_{FP}$, $\mathcal{O}_F|_X$ induces the original topology on $X$. Thus we satisfied the conditions of Definition~\ref{freeM1}, which implies that $\mathcal{O}_F=\mathcal{T}_{FP}$. Therefore, $\mathcal{N}_F$ is a neighborhood base at $e$ of the group $\FP(X)$.
\end{proof}
Now let $\mathscr{H}_F=\{gN_F: g\in F_a(X)\}$ and let $\mathscr{H}_A=\{g+N_A: g\in A_a(X)\}$. If $g_1, g_2\in \FP(X)$ such that $g_1\in g_2N_F$, then we have $g_1N_F\subseteq g_2N_FN_F=g_2N_F$. A similar result is true for~$\mathscr{H}_A$.\\

Let $X$ be a set. Then for all $k\in \Bbb Z$, we define $Z_k(X)=\{x_1^{\epsilon_1}x_2^{\epsilon_2}\ldots x_n^{\epsilon_n}\in F_a(X): \sum_{i=1}^n \epsilon_i=k\}$ and $Z^A_k(X)=\{\epsilon_1x_1+\epsilon_2x_2+\cdots +\epsilon_nx_n\in A_a(X): \sum_{i=1}^n \epsilon_i=k\}$.
For every $k_1, k_2\in \Bbb Z$ and $k_1\neq k_2$, the sets $Z_{k_1}(X)$ and $Z_{k_2}(X)$ are disjoint and the sets $Z^A_{k_1}(X)$ and $Z^A_{k_2}(X)$ are disjoint. The set $Z_0(X)$ is the smallest normal subgroup of $F_a(X)$ containing the set $Z_F=\bigcup_{x\in X}x^{-1}X$ and the set $Z^A_0(X)$ is the smallest subgroup of $A_a(X)$ containing the set $Z_A=\bigcup_{x\in X}X-x$.\\

Let $X$ be a topological space and let $I\colon X\to \AP(X)$ be the identity mapping of the space $X$ to the abelian group $\AP(X)$. Then we extend $I$ to the continuous homomorphism  mapping $\hat{I}\colon FP(X)\to AP(X)$. We call the mapping $\hat{I}$ the \textit{canonical mapping}.

\begin{theorem}
\label{Z0}
Let $X$ be an Alexandroff space. Then the following are equivalent.
\begin{enumerate}
\item[(1)] The space $X$ is indiscrete.
\item[(2)] $N_F=Z_0(X)$ in $FP(X)$.
\item[(3)] $N_A=Z_0^A(X)$ in $\AP(X)$.
\end{enumerate}
\end{theorem}
\begin{proof}
(1)$\Rightarrow$(2): Assume that $X$ is indiscrete. Then $U(x)=X$ for all $x\in X$ and so $U_F=Z_F$, where $Z_F$ is the generating set for $Z_0(X)$. Therefore, $N_F=Z_0(X)$.

(2)$\Rightarrow$(3): Assume that $N_F=Z_0(X)$. Let $\hat{I}\colon \FP(X)\to \AP(X)$ be the canonical mapping. Thus $\hat{I}(N_F)=\hat{I}(Z_0(X))$. Since $\hat{I}(N_F)=N_A$ and $\hat{I}(Z_0(X))=Z_0^A(X)$, so $N_A=Z_0^A(X)$.

(3)$\Rightarrow$(1): Assume that $N_A=Z_0^A(X)$. Thus $Z_k^A(X)$ is open in $\AP(X)$ for each $k\in \Z$. Since $Z_1^A(X)\cap X=X$ and $Z_k^A(X)\cap X=\emptyset$ for all $k\in \Z\setminus \{1\}$, we have $X$ is indiscrete.
\end{proof}

 We call a space $X$ a \textit{partition} space if $X$ has a base which is a partition of $X$. Clearly, every partition space is an Alexandroff space.

It is easy to see that if $X$ is a partition space, then the collection $\{U(x)\}_{x\in X}$ is a partition on $X$.

\begin{theorem}
If $X$ is a partition space, then the free paratopological groups $\FP(X)$ and $\AP(X)$ are partition spaces.
\end{theorem}
\begin{proof}
Let~$X$ be a partition space. Then $N_F$ and $N_A$ are normal subgroups of $\FP(X)$ and $\AP(X)$, respectively. Thus the collections~$\mathscr{H}_F$ and~$\mathscr{H}_A$ as defined above are partitions of~$\FP(X)$ and~$\AP(X)$, respectively. Therefore the result follow.
\end{proof}

\section{Applications}
Let $\mathcal{T}_A$ be the topology of the subspace $X^{-1}$ of $\FP(X)$, where $X$ be any topological space. Then by Theorem 4.2 of~\cite{Elfard1}, the topology $\mathcal{T}_A$ has as an open base the collection $\{C^{-1}: C$ closed in $X\}$. In this topology, the intersection of every collection of open subsets is open, and the space $X_A^{-1} = (X^{-1}, \mathcal{T}_A)$ is therefore an Alexandroff space.

\begin{theorem}
\label{topgp}
Let $X$ be a topological space. Then the group $\FP(X)$ on $X$ is a topological group if and only if $X$ is a partition space.
\end{theorem}
\begin{proof}
$\Longrightarrow$: Assume that $\FP(X)$ is a topological group. Let $U$ be an open set in $X$. By the argument above, the topology on the subspace $X_A^{-1}$ of $\FP(X)$ has the collection $\{C^{-1}: C$ is closed in $X\}$ as a base. Thus $(U^c)^{-1}$ is open in $X_A^{-1}$. Since the inversion mapping of $X^{-1}$ to $X$ is a homeomorphism, $U^c$ is open in $X$. So $U$ is closed in $X$ and therefore, $X$ is a partition space.

$\Longleftarrow$: If $X$ is a partition space, then $N_F$ is a subgroup of $\FP(X)$. Therefore, $\FP(X)$ is a topological group.

The same proof works for $\AP(X)$.
\end{proof}

\begin{proposition}
 \label{w}
 Let $X$ be an Alexandroff space and let $\FP(X)$ be the free paratopological group on $X$. Then the space $X$ is $T_0$ if and only if for each $w\in N_F$ and $w\neq e$ we have  $\hat{I}(w)\neq 0_A$.
\end{proposition}
\begin{proof}
$\Longrightarrow$: Suppose that there exists $w\in N_F$ and $w\neq e$ such that $\hat{I}(w)=0_A$, where
 \begin{center}$w=g_1x_1^{-1}y_1g_1^{-1}g_2x_2^{-1}y_2g_2^{-1}\cdots g_nx_n^{-1}y_ng_n^{-1}$ for some $n\in\N, y_i\neq x_i$ and $y_i\in U(x_i)$ for all $i=1, 2, \ldots, n$.\end{center} Then we have $\hat{I}(w)=y_1-x_1+y_2-x_2+\cdots +y_n-x_n=0_A$.
If $n=1$, then $x_1=y_1$, which gives a contradiction. Assume that $n>1$. Since $\hat{I}(w)=0_A$, for each $i\in A=\{1, 2, \ldots, n\}$, there exists $j_i\in A$, where $j_i\neq i$ such that $x_i=y_{j_i}$. Define $\sigma\colon A\to A$ by setting $\sigma(i)=j_i$ for all $i\in A$. Clearly $\sigma$ is a permutation on $A$. Since any permutation can be written as  product of cycles, there are $m\in \N$, where $2\leq m \leq n$ and distinct $i_1, i_2, \ldots, i_m\in A$ such that $\sigma(i_1)=i_2$, $\sigma(i_2)=i_3, \dots, \sigma(i_{m-1})=i_m, \sigma(i_m)=i_1$ and such that $x_{i_k}=y_{\sigma(i_k)}$ for $k=1, 2, \ldots, m$. Thus $x_{i_1}=y_{i_2}, x_{i_2}=y_{i_3}, \ldots, x_{i_{m-1}}=y_{i_m}, x_{i_m}=y_{i_1}$ and hence
$$U(y_{i_1})\subseteq U(x_{i_1})=U(y_{i_2})\subseteq U(x_{i_2})=U(x_{i_3})\subseteq \cdots $$
$$\subseteq U(x_{i_{m-1}})=U(y_{i_m})\subseteq U(x_{i_m})=U(y_{i_1}),$$
which implies that $$U(y_{i_1})=U(x_{i_1})=U(y_{i_2})=U(x_{i_2})=\cdots =U(x_{i_{m-1}})=U(y_{i_m}).$$
Thus we can not separate the points $y_{i_1}, x_{i_1}, y_{i_2}, x_{i_2}, \ldots, x_{i_{m-1}}, y_{i_m}$. Therefore, $X$ is not a $T_0$~space.

$\Longleftarrow$: Assume that $X$ is not $T_0$. Then there are $x, y\in X$ such that $x\neq y$ and $U(x)=U(y)$, which implies that $x\in U(y)$ and $y\in U(x)$. Hence $$w=x^{-1}yxy^{-1}=(x^{-1}y)(x(y^{-1}x)x^{-1})\in N_F$$ and $w\neq e$, but $\hat{I}(w)=0_A$. Therefore, the space~$X$ is~$T_0$.
\end{proof}

We note that a corollary of this result is that if $X$ is an Alexandroff $T_0$~space, then $\hat{I}$ has the property that $\ker\hat{I}\cap N_F=\{e\}$.

The following result is easy to prove.

\begin{proposition}
 \label{G}
 Let $G$ be a paratopological group. Then $G$ is a $T_0$~space if and only if for all $a\in G$ such that $a\neq e$, there exists a neighborhood $U$ of $e$ such that either $a\notin $U or $a^{-1}\notin U$.
\end{proposition}

\begin{proposition}
\label{AT0}
Let $X$ be an Alexandroff space. Then $\FP(X)$ is a $T_0$~space if and only if $X$ is a $T_0$~space.
\end{proposition}
\begin{proof}
$\Longrightarrow$: Since $X$ is a subspace of $\FP(X)$, the result follows.

$\Longleftarrow$: Assume that $X$ is $T_0$. We claim that $\FP(X)$ is $T_0$. In fact, if $\FP(X)$ is not $T_0$, then by Proposition~\ref{G}, there exists $w\in \FP(X)$, $w\neq e$ such that $w, w^{-1}\in N_F$. Hence by Proposition~\ref{w}, $\hat{I}(w)\neq 0_A$ and it is easy to see that $\hat{I}(w), -\hat{I}(w)\in N_A$, which implies that $\hat{I}(w), -\hat{I}(w)$ are in every neighborhood of $0_A$. Once again by Proposition~\ref{G}, $\AP(X)$ is not $T_0$ and so by Proposition 3.4 of~\cite{Pyrch1} (which says that if a space $X$ is $T_0$, then $\AP(X)$ is $T_0$), $X$ is not a $T_0$~space, which gives a contradiction. Therefore, $\FP(X)$ is $T_0$.
\end{proof}

 Fix $n \in \N$ and let $R_n = \{1, 2, \ldots, n\}\subseteq \N$. For $i = 0, 1, \ldots, n$, define $R_{n,i} = \{1, 2, \ldots, i\}$ and  $\tau_n = \{R_{n,i} : i = 0, \ldots, n\}$. Then it is easy to see that $\tau_n$ is a topology on $R_n$. Let $m, k\in R_n$, where $m\neq k$ and assume that $m<k$. Then $m\in R_{n, m}$ and $k\notin R_{n, m}$. Therefore,  $(R_n, \tau_n)$ is a $T_0$~space.

\begin{proposition}
\label{Xn}
Let $X$ be a $T_0$~space and let $x_1, x_2,\ldots, x_n$ be distinct elements of~$X$. Then there exists a continuous mapping $\mu\colon X \to R_n$ such that $\mu|_{\{x_1, x_2, \ldots, x_n\}}$ is one-to-one.
\end{proposition}

\begin{theorem}
\label{T0}
Let $X$ be a topological space. Then the free paratopological group $\FP(X)$ on $X$ is $T_0$ if and only if the space $X$ is $T_0$.
\end{theorem}
\begin{proof}

$\Longrightarrow$: It is clear.

$\Longleftarrow$: Let $w=x_1^{\epsilon_1}x_2^{\epsilon_2}\cdots x_m^{\epsilon_m}\in \FP(X)$ for some $m\in \N$ such that $w\neq e$. Choose indices $i_1, i_2, \ldots, i_n$ for some $n\leq m$ such that $x_{i_1}, x_{i_2}, \ldots, x_{i_n}$ are the distinct letters among $x_1, x_2, \ldots, x_m$. Then by Proposition~\ref{Xn}, there exists a continuous mapping $\mu:X\to R_n$ such that $\mu|_{\{x_{i_1}, x_{i_2}, \ldots, x_{i_n}\}}$ is one-to-one, where $R_n$ is the space defined above. Then we extend $\mu$ to a continuous homomorphism $\hat{\mu}:\FP(X)\to \FP(R_n)$. Since $\mu|_{\{x_{i_1}, \ldots, x_{i_n}\}}$ is one-to-one, $\hat{\mu}(w)=[\hat{\mu}(x_1)]^{\epsilon_1}[\hat{\mu}(x_2)]^{\epsilon_2}\cdots [\hat{\mu}(x_n)]^{\epsilon_n}\neq e^*$, where $e^*$ is the identity of $\FP(R_n)$. By Proposition~\ref{AT0}, we have  $\FP(R_n)$ is a $T_0$~space. So there is an open set $U$ in $\FP(R_n)$, which contains one of $e^*$ or $\hat{\mu}(w)$ and does not contain the other. Say $e^*\in U$ and $\hat{\mu}(w)\notin U$. Since $\hat{\mu}$ is continuous, $\hat{\mu}^{-1}(U)$ is an open set in $\FP(X)$ such that $e\in \hat{\mu}^{-1}(U)$ and $w\notin \hat{\mu}^{-1}(U)$. Similarly for the other case. Therefore, the free paratopological group $\FP(X)$ is~$T_0$.
\end{proof}

 A topological space~$X$ is said to be the \textit{inductive limit of a cover $\mathscr{C}$} of $X$ if a subset $V$ of~$X$ is open whenever $V\cap U$ is open in $U$ for each~$U\in \mathscr{C}$.

A parallel result of the next theorem was proved in Proposition 7.4.8 of~\cite{Arhan} in the case of free topological groups.

\begin{theorem}
\label{induc}
Let $X$ be a $T_1$ $P$-space. Then the free paratopological group $\FP(X)$ ($\AP(X)$) is the inductive limit of the collection $\{\FP_n(X)\colon n\in \N\}$ ($\{\AP_n(X)\colon n\in \N\}$).
\end{theorem}
\begin{proof}
We prove the statement for $\FP(X)$, since the proof for $\AP(X)$ is similar. Let $C$ be a subset of $\FP(X)$ such that $C\cap \FP_n(X)$ is closed in $\FP_n(X)$ for all $n\in\N$. By Theorem~4.1.3 of~\cite{Elfard1}, the sets $\FP_n(X)$ are closed in $\FP(X)$ for all $n\in\N$. Thus the sets $C\cap \FP_n(X)$ are closed in $\FP(X)$ for all $n\in\N$, which implies that $C$ is a countable union of closed sets in $\FP(X)$. Since the group $\FP(X)$ is a $P$-space, $C$ is closed in $\FP(X)$ and then $\FP(X)$ is the inductive limit of the collection $\{\FP_n(X)\colon n\in \N\}$.
\end{proof}

The author wishes to thank associate professor Peter Nickolas for helpful comments.

\def\cprime{$'$} \def\cprime{$'$} \def\cprime{$'$} \def\cprime{$'$}
  \def\cprime{$'$}
\providecommand{\bysame}{\leavevmode\hbox to3em{\hrulefill}\thinspace}

\end{document}